\documentclass[a4paper,12pt]{amsart}
\usepackage{amssymb}
\usepackage{ifthen}
 \usepackage[dvips]{graphicx}
\nonstopmode \numberwithin{equation}{section}
\usepackage{amssymb}
\usepackage{ifthen}
\usepackage{graphicx}
\usepackage{amsmath}
\usepackage[T1]{fontenc} 
\usepackage[utf8]{inputenc}
\usepackage[usenames,dvipsnames]{color}
\usepackage{color}
\usepackage[english]{babel}
\usepackage{fancyhdr}
\usepackage{fancybox}
\usepackage{tikz}

\nonstopmode \numberwithin{equation}{section}
\theoremstyle{plain}

\newtheorem{conj}{Conjecture}

\theoremstyle{definition}
\newtheorem{defn}{Definition}[section]

\newtheorem{thm}{Theorem}[section]
\newtheorem{prob}{Problem}[section]
\newtheorem{cor}{Corollary}[section]
\newtheorem{ques}{Question}[section]
\newtheorem{prop}{Proposition}[section]
\newtheorem{rem}{Remark}[section]
\newtheorem{lem}{Lemma}[section]

\newcounter{minutes}
\divide\time by 60
\newcounter{hours}
\multiply\time by 60
\addtocounter{minutes}{-\time}

\newcounter {own}
\def\theown {\thesection       .\arabic{own}}

\newenvironment{pf}[1][]{%
 \vskip 3mm
 \noindent
 \ifthenelse{\equal{#1}{}}%
  {{\slshape Proof. }}%
  {{\slshape #1.} }%
 }%
{\qed\bigskip}

\newcounter{alphabet}

\def\be{\begin{equation}}
\def\ee{\end{equation}}

\newcommand{\bee}{\begin{enumerate}}
\newcommand{\eee}{\end{enumerate}}

\newcommand{\blem}{\begin{lem}}
\newcommand{\elem}{\end{lem}}
\newcommand{\bthm}{\begin{thm}}
\newcommand{\ethm}{\end{thm}}
\newcommand{\bcor}{\begin{cor}}
\newcommand{\ecor}{\end{cor}}
\newcommand{\beg}{\begin{examp}}
\newcommand{\eeg}{\end{examp}}
\newcommand{\begs}{\begin{examples}}
\newcommand{\eegs}{\end{examples}}

\newcommand{\bdefn}{\begin{defn}}
\newcommand{\edefn}{\end{defn}}

\newcommand{\bprob}{\begin{prob}}
\newcommand{\eprob}{\end{prob}}
\newcommand{\bei}{\begin{itemize}}
\newcommand{\eei}{\end{itemize}}

\newcommand{\bcon}{\begin{conj}}
\newcommand{\econ}{\end{conj}}
\newcommand{\bcons}{\begin{conjs}}
\newcommand{\econs}{\end{conjs}}
\newcommand{\bprop}{\begin{prop}}
\newcommand{\eprop}{\end{prop}}
\newcommand{\br}{\begin{rem}}
\newcommand{\er}{\end{rem}}
\newcommand{\brs}{\begin{rems}}
\newcommand{\ers}{\end{rems}}
\newcommand{\bo}{\begin{obser}}
\newcommand{\eo}{\end{obser}}
\newcommand{\bos}{\begin{obsers}}
\newcommand{\eos}{\end{obsers}}
\newcommand{\bpf}{\begin{pf}}
\newcommand{\epf}{\end{pf}}
\newcommand{\ba}{\begin{array}}
\newcommand{\ea}{\end{array}}
\newcommand{\beq}{\begin{eqnarray}}
\newcommand{\beqq}{\begin{eqnarray*}}
\newcommand{\eeq}{\end{eqnarray}}
\newcommand{\eeqq}{\end{eqnarray*}}

\begin{document}

\title{Generalized weighted composition operators  on the hardy space $H^2(\mathbb{D}^n)$}

\author{Molla Basir Ahamed}
\address{Molla Basir Ahamed, Department of Mathematics, Jadavpur University, Kolkata-700032, West Bengal, India.}
\email{mbahamed.math@jadavpuruniversity.in}

\author{Vasudevarao Allu}
\address{Vasudevarao Allu,
	School of Basic Science,
	Indian Institute of Technology Bhubaneswar,
	Bhubaneswar-752050, Odisha, India.}
\email{avrao@iitbbs.ac.in}

\author{Taimur Rahman}
\address{Taimur Rahman, Department of Mathematics, Jadavpur University, Kolkata-700032, West Bengal, India.}
\email{taimurr.math.rs@jadavpuruniversity.in}

\subjclass[{AMS} Subject Classification:]{Primary 47B33, 47B02 Secondary 47A05, 47B15}

\def\thefootnote{}
\footnotetext{ {\tiny File:~\jobname.tex,
printed: \number\year-\number\month-\number\day,
          \thehours.\ifnum\theminutes<10{0}\fi\theminutes }
} \makeatletter\def\thefootnote{\@arabic\c@footnote}\makeatother

\begin{abstract} 
In this paper, we explore the complex symmetrical characteristics of weighted composition operators $W_{u, v}$ and weighted composition-differentiation operators $W_{u, v, k_1, k_2, \ldots, k_n}$ on the Hardy space $H^2(\mathbb{D}^n)$ over the Polydisk $\mathbb{D}^n$, with respect to the standard conjugation $\mathcal{J}$. We specify explicit conditions that confirm the Hermitian characteristics of the operator $W_{u, v, k_1, k_2, \ldots, k_n}$ and describe the conditions necessary for it to exhibit normal behavior. Additionally, we identify the kernels of the generalized weighted composition-differentiation operators and their corresponding adjoint operators.

\end{abstract}
\maketitle
\pagestyle{myheadings}
\markboth{M. B. Ahamed, V. Allu and T. Rahman}{Generalized weighted composition operators on Hardy space  $H^2(\mathbb{D}^n)$}
\section{Introduction}
Let $\mathbb{T}$ denote the boundary of the unit disk $\mathbb{D}:=\{z\in\mathbb{C}:|z|<1\}$ in the complex plane $\mathbb{C}$. The polydisk $\mathbb{D}^n$ and the torus $\mathbb{T}^n$ can be represented as the Cartesian products of $\mathbb{D}$  and $\mathbb{T}$, respectively. Let $L^2:=L^2(\mathbb{T}^n, \sigma)$  denote the standard Lebesgue space over $\mathbb{T}^n$, where $\sigma$ stands for the normalized Haar measure on the torus $\mathbb{T}^n$. It has been established that the set $\{e_{m_1, m_2, \ldots, m_n}(z_1,z_2, \ldots, z_n)=\prod_{i=1}^{n}z_{i}^{m_i}: m_1, m_2, \ldots, m_n\in\mathbb{Z}\}$ forms an orthonormal basis for $L^2$. If a function $f$ belongs to $L^2$, it can be expressed by using its Fourier coefficients $ \Hat{f}(m_1, m_2, \ldots, m_n) $ as
\begin{align*}
 f(z_1, z_2, \ldots, z_n)=\sum_{m_1, m_2, \ldots, m_n=-\infty}^{\infty}\Hat{f}(m_1, m_2, \ldots, m_n)\prod_{i=1}^{n}z^{m_i}_i.	
\end{align*}
Furthermore, we define $L^\infty = L^\infty(\mathbb{T}^n)$ as the Banach space containing all essentially bounded functions defined on $\mathbb{T}^n$.\vspace{1.2mm}
 
 The Hilbert space denoted as $H^2=H^2(\mathbb{D}^n)$, referred to as the Hardy space, consists of all analytic functions $f$ defined on the polydisk $\mathbb{D}^n$, and satisfies the condition
 \begin{align*}
 	 ||f||^2 := \sup_{0 < r < 1} \int_{\mathbb{T}^n} |f(r\zeta)|^2 \, d\sigma(\zeta) < \infty.
 \end{align*}
It is a well-known fact that for a function $f$ belonging to the Hardy space $H^2(\mathbb{D}^n)$, there exists a boundary function $\lim_{r\to 1}f(r\xi)$, which is an element of $L^2$, for almost every $\xi\in\mathbb{T}^n$. The Poisson extension of this boundary function coincides with $f$ on the polydisk $\mathbb{D}^n$. By identifying a function in $H^2(\mathbb{D}^n)$ with its boundary function, we can consider $H^2(\mathbb{D}^n)$ as a subset of $L^2$. Hence, the norm of $f\in H^2$ can be denoted as 
\begin{align*}
		||f||^2=\int_{\mathbb{T}^n}|f(\zeta)|^2 d\sigma(\zeta).
\end{align*}
The reproducing kernel function for the Hardy space $H^2$ is defined as
\begin{align*}
	K_{\alpha}(z)=K_{\alpha_1, \alpha_2, \ldots, \alpha_n}(z_1, z_2, \ldots, z_n)=\prod_{i=1}^{n}K_{\alpha_i}(z_i)=\prod_{i=1}^{n}\frac{1}{(1-\bar{\alpha_i}z_i)},
\end{align*}
where $z=(z_1, z_2, \ldots, z_n), \alpha=(\alpha_1, \alpha_2, \ldots,\alpha_n)\in\mathbb{D}^n$, and $K_{\alpha_i}(z_i)$'s are the reproducing kernel functions for $H^2(\mathbb{D})$. The normalized reproducing kernel for $H^2(\mathbb{D}^n)$ is denoted by $k_{\alpha}(z) =\prod_{i=1}^{n}e_{\alpha_i}(z_i)$, with $e_{\alpha_1}(z_1), e_{\alpha_2}(z_2), \ldots, e_{\alpha_n}(z_n)$ being the normalized reproducing kernel functions for $H^2(\mathbb{D})$.\vspace{2mm}
 
Let $ \mathcal{B}(\mathcal{H}) $ be the algebra of all bounded linear operators on a separable complex Hilbert space $ \mathcal{H} $. A conjugation on $ \mathcal{H} $ is an antilinear operator $ \mathcal{C} : \mathcal{H} \to \mathcal{H}$ which satisfies $ \langle \mathcal{C}x, \mathcal{C}y \rangle_{\mathcal{H}}=\langle y, x \rangle_{\mathcal{H}} $ for all $ x, y\in\mathcal{H} $ and $ \mathcal{C}^2=I_{\mathcal{H}} $, where $ I_{\mathcal{H}} $ is the identity operator on the Hilbert space $ {\mathcal{H}} $. An operator $ T\in\mathcal{B}({\mathcal{H}}) $ is said to be complex symmetric if there exists a conjugation $ \mathcal{C} $ on $ \mathcal{H} $ such that $ T=\mathcal{C}T^*\mathcal{C} $. In this case, we say that $ T $ is complex symmetric with conjugation $ \mathcal{C} $. The comprehensive examination of complex symmetric operators was initiated by Garcia, Putinar, and Wogen \cite{Garcia-Putinar-TAMS-2006,Garcia-Putinar-TAMS-2007,Garcia-Putinar-TAMS-2008,Garcia-Wogen-JFA-2009,Garcia-Wogen-TAMS-2010}. The class of complex symmetric operators includes normal operators, binormal operators, operators with algebraic degree two, Hankel operators, truncated Toeplitz operators, and the Volterra integration operator. In the recent times, there has been a considerable amount of research devoted to investigate the complex symmetric composition operators that operate on classical Hilbert space of analytic functions (see \cite{Bourdon-Noor-JMAA-2015,Garcia-Hammond-OPAP-2013,Jun-Kim-Ko-Lee-JFA-2014,Lim-Khoi-JMAA-2018,Noor-Severiano-PAMS-2020,Yao-JMAA-2017} and references therein).\vspace{2mm}

In this context, $W_{u, v}$ refers to the weighted composition operator, while $W_{u, v, k_1, k_2, \ldots, k_n}$ represents the weighted composition-differentiation operator on $H^2(\mathbb{D}^n$). The problem of determining when such an operator acts as a conjugation on the weighted Hardy space of the unit disk was fully resolved in \cite{Lim-Khoi-JMAA-2018}. Moreover, this finding was extended to the unit ball in \cite{Han-Wang-JMAA-2019, Hu-Yang-Zhou-IJM-2020, Wang-Yao-IJM-2016}. Our paper will specifically concentrate on the unit polydisk scenario. However, as we move to higher dimensions, we confront not just combinatorial complexities but also unexpected barriers and novel phenomena. For instance, studying the boundedness of the operator $W_{u, v}$ in $H^2(\mathbb{D}^n)$ becomes more complicated and challenging.\vspace{1.2mm} 

In general, achieving a complete characterization of all complex symmetric composition operators is difficult, encouraging researchers to investigate complex symmetric weighted composition operators with specific conjugation traits. In this paper, we delve into the analysis of complex symmetric weighted composition operators using the standard conjugation $\mathcal{J}$.\vspace{1.2mm}

The organization of this paper is as follows. In section 2, the operators $W_{u, v}$ and $W_{u, v, k_1, k_2, \ldots, k_n}$ are shown to possess complex symmetry on $H^2(\mathbb{D}^n)$ as we explicitly determine the corresponding forms of $u$ and $v_j$ (where $j=1, 2, \ldots, n$), taking into account the standard conjugation $\mathcal{J}$. In Section 3, we establish conditions that are both necessary and sufficient for $W_{u, v, k_1, k_2, \ldots, k_n}$ to exhibit Hermitian characteristics. Section 4 deals with a sufficient condition for the normality of $W_{u, v, k_1, k_2, \ldots, k_n}$. In section 5, we investigate the kernels of both $W_{u, v, k_1, k_2, \ldots, k_n}$ and its adjoint operator $W^*_{u, v, k_1, k_2, \ldots, k_n}$.
\section{Complex symmetric weighted composition-differentiation operator on Hardy space $ H^2(\mathbb{D}^n) $}
Let $u :\mathbb{D}^n\to \mathbb{C}$ be an analytic function and $v :\mathbb{D}^n\to\mathbb{D}^n$ be an analytic self-map of the polydisk $\mathbb{D}^n$ given by $v(z_1, z_2, \ldots, z_n) = (v_1(z_1), v_2(z_2), \ldots, v_n(z_n))$, where $v_1, v_2, \ldots, v_n$ are analytic self-maps of the unit disk $\mathbb{D}$. We now define the weighted composition operator $W_{u, v}$ on the Hardy space $H^2(\mathbb{D}^n)$ as follows:
\begin{align*}
	W_{u, v}f(z)=u(z)f(v(z)),\; f\in H^2(\mathbb{D}^n),
\end{align*}
that is,
\begin{align*}
W_{u, v}f(z_1, z_2, \ldots, z_n)=u(z_1, z_2, \ldots, z_n)f(v_1(z_1), v_2(z_2), \ldots, v_n(z_n)),\; f\in H^2(\mathbb{D}^n).
\end{align*}
The main aim of this paper is to explore different characteristics of weighted composition-differentiation operators in the Hardy space $H^2(\mathbb{D}^n)$, exceeding current research outcomes. Going forward, we shall refer to the composition-differentiation operator as $W_{u, v, k_1, k_2, \ldots, k_n}$, defined as
\begin{align*}
W_{u, v, k_1, k_2, \ldots, k_n}f(z)=u(z)\frac{\partial^{\sum_{i=1}^{n}k_i}f}{\partial z^{k_1}_1\partial z^{k_2}_2 \ldots \partial z^{k_n}_n}(v(z))
\end{align*}
where $f$ is an element of the Hardy space $H^2(\mathbb{D}^n)$. It is clear and straightforward that $W_{u, v, 0, 0, \ldots, 0}=W_{u, v}$. \vspace{1.2mm}

Assuming $k_1, k_2, \ldots, k_n$ are non-negative integers and $w\in\mathbb{D}^n$, we define $K_w^{[k_1, k_2, \ldots, k_n]}(z)$ as the reproducing kernel for evaluating the $\sum_{i=1}^{n}k_i$-th partial derivative at a specific point $w$, where
\begin{align*}
	K_w^{[k_1, k_2, \ldots, k_n]}(z)=\prod_{i=1}^{n}\frac{k_i!z^{k_i}i}{(1-\bar{w_i}z_i)^{k_i+1}}, z\in\mathbb{D}^n.
\end{align*}
It can be easily verified that $\bigg\langle f(z), K_w^{[k_1, k_2, \ldots, k_n]}(z)\bigg\rangle=\frac{\partial^{\sum_{i=1}^{n}k_i}f}{\partial z^{k_1}_1\partial z^{k_2}_2 \ldots \partial z^{k_n}_n}(w)$.\vspace{1.2mm}

Here the weighted composition conjugation on $H^2(\mathbb{D}^n)$ is of the form 
\begin{align*}
	A_{u, v}f(z)=u(z)\overline{f(\overline{v(z)})}, f\in\ H^2(\mathbb{D}^n)
\end{align*}
that is,
\begin{align*}
		A_{u, v}f(z_1, z_2, \ldots, z_n)=u(z_1, z_2, \ldots, z_n)\overline{f(\overline{v_1(z_1)}, \overline{v_2(z_2)}, \ldots, \overline{v_n(z_n)})}, f\in\ H^2(\mathbb{D}^n).
\end{align*}
If the function $u$ is identically equal to $1$, and $v$ is the identity map of the polydisk $\mathbb{D}^n$, the weighted composition conjugation $A_{u, v}$ simplifies to the standard conjugation $\mathcal{J}$, defined as $\mathcal{J}f(z):=\overline {f(\bar{z})}$. One can easily see that $A_{u, v}=W_{u, v}\mathcal{J}$. For a more comprehensive understanding of weighted composition conjugations, we refer to \cite{Lim-Khoi-JMAA-2018} and references therein.\\\vspace{1.2mm} 
In the following, we provide two lemmas that are focused on the action of the adjoint of the composition operator $W_{u, v}$ and  composition-differentiation operator $W_{u, v, k_1, k_2, \ldots, k_n}$ on the reproducing kernel of $H^2(\mathbb{D}^n)$. In the entirety of the article, we will use the notation $m=\sum_{i=1}^{n}{k_i}$.

\begin{lem}\label{lem-2.1}
Let $u:\mathbb{D}^n\to \mathbb{C} $ be a nonzero analytic function and, let $v:\mathbb{D}^n\to\mathbb{D}^n $ with $v(z_1, z_2, \ldots, z_n)=(v_1(z_1), v_2(z_2),\ldots, v_n(z_n))$ be an analytic self-map of $\mathbb{D}^n$ such that $W_{u, v}$ is bounded on $H^2(\mathbb{D}^n)$. Then for any $ w\in\mathbb{D}^n$, 
	\begin{align*}
		W^*_{u, v,}K_w(z)=\overline{u(w)}K_{v(w)}(z).
	\end{align*}
\end{lem}
\begin{proof}[\bf Proof of Lemma \ref{lem-2.1}]
	Let $ f\in H^2(\mathbb{D}^n)$. Then a simple computation shows that 
	\begin{align*}
		\langle f, W^*_{u, v}K_w\rangle=&\langle W_{u, v}f, K_w\rangle\\=&\langle uf(v), K_w\rangle\\=&u(w)f(v(w))\\=&\langle f, \overline{u(w)}K_w\rangle.
	\end{align*} 
	Therefore, we have 
	\begin{align*}
		\langle f, W^*_{u, v}K_w\rangle=\langle f, \overline{u(w)}K_w\rangle\;\mbox{for all}\; f\in\ H^2(\mathbb{D}^n)	
	\end{align*}
	which leads the following conclusion
	\begin{align*}
	W^*_{u, v,}K_w(z)=\overline{u(w)}K_{v(w)}(z).	
	\end{align*}
	This completes the proof.
\end{proof}
\begin{lem}\label{lem-2.2}
Let $u:\mathbb{D}^n\to \mathbb{C} $ be a nonzero analytic function and, let $v:\mathbb{D}^n\to\mathbb{D}^n $ with $v(z_1, z_2, \ldots, z_n)=(v_1(z_1), v_2(z_2),\ldots, v_n(z_n))$ be an analytic self-map of $\mathbb{D}^n$ such that $W_{u, v, k_1, k_2, \ldots, k_n}$ is bounded on $H^2(\mathbb{D}^n)$. Then for any $ w\in\mathbb{D}^n$,
\begin{align*}
	W^*_{u, v, k_1, k_2, \ldots, k_n}K_w(z)=\overline{u(w)}K^{[k_1, k_2, \ldots, k_n]}_{v(w)}(z)
\end{align*} 	
\end{lem}
\begin{proof}[\bf Proof of Lemma \ref{lem-2.2}]
		Let $ f\in H^2(\mathbb{D}^n)$. Then a simple computation shows that 
	\begin{align*}
		\langle f, W^*_{u, v, k_1, k_2, \ldots, k_n}K_w\rangle=&\langle W_{u, v, k_1, k_2, \ldots, k_n}f, K_w\rangle\\=&\left\langle u\frac{\partial^mf}{\partial z^{k_1}_1\partial z^{k_2}_2 \ldots \partial z^{k_n}_n}(v), K_w\right \rangle\\=&u(w)\frac{\partial^mf}{\partial z^{k_1}_1\partial z^{k_2}_2 \ldots \partial z^{k_n}_n}(v(w))\\=&\left\langle f, \overline{u(w)}K^{[k_1, k_2, \ldots, k_n]}_{v(w)}\right\rangle.
	\end{align*} 
	Therefore, we have 
	\begin{align*}
		\langle f, W^*_{u, v, k_1, k_2, \ldots, k_n}K_w\rangle=\left\langle f, \overline{u(w)}K^{[k_1, k_2, \ldots, k_n]}_{v(w)}\right\rangle\;\mbox{for all}\; f\in\ H^2(\mathbb{D}^n)	
	\end{align*}
	which leads the following conclusion
	\begin{align*}
	W^*_{u, v, k_1, k_2, \ldots, k_n}K_w(z)=\overline{u(w)}K^{[k_1, k_2, \ldots, k_n]}_{v(w)}(z).	
	\end{align*}
	This completes the proof.  
\end{proof}
In the following result, we are aiming to examine the circumstances under which specific combinations of $u$ and $v_j(j=1, 2, \ldots, n)$ give rise to $\mathcal{J}$-symmetric weighted composition-differentiation operators $W_{u, v}$.

\begin{thm}\label{thm-2.1}
	Let $u:\mathbb{D}^n\to \mathbb{C} $ be a nonzero analytic function and, let $v:\mathbb{D}^n\to\mathbb{D}^n $ with $v(z_1, z_2, \ldots, z_n)=(v_1(z_1), v_2(z_2),\ldots, v_n(z_n))$ be an analytic self-map of $\mathbb{D}^n$ such that $W_{u, v}$ is bounded on $H^2(\mathbb{D}^n)$. Then $W_{u, v}$ is $\mathcal{J}$-symmetric if, and only if,
	\begin{align*}
		u(z_1, z_2, \ldots, z_n)=a\prod_{i=1}^{n}\frac{1}{(1-c_iz_i)}\;\mbox{and}\; v_j(z_j)=\frac{(d_j-c_j^2)z_j+c_j}{1-c_jz_j}\;\; \mbox{for}\; j=1, 2, \ldots, n,
	\end{align*}
	where $ a=u(0, 0, \ldots, 0) $ and $c_j=v_j(0)$, $d_j=v^\prime(0)$ for $j=1, 2, \ldots, n $.
\end{thm}

If we substitute the polydisk $\mathbb{D}^n$ with the bidisk $\mathbb{D}^2$, then Theorem \ref{thm-2.1} becomes a more comprehensive version of \cite[Theorem 2.1]{Han-Wang-Wu-AMS-2023}, which Han $et$ $al$ \cite{Han-Wang-Wu-AMS-2023}  recently established. Consequently, we can easily derive the following result.
\begin{cor}
Let $u:\mathbb{D}^2\to \mathbb{C} $ be a nonzero analytic function and, let $v:\mathbb{D}^2\to\mathbb{D}^2 $ with $v(z_1, z_2)=(v_1(z_1), v_2(z_2))$ be an analytic self-map of $\mathbb{D}^2$ such that $W_{u, v}$ is bounded on $H^2(\mathbb{D}^2)$. Then $W_{u, v}$ is $\mathcal{J}$-symmetric if, and only if,
	\begin{align*}
	u(z_1, z_2)=\frac{a}{(1-cz_1)(1-dz_2)},\; v_1(z_1)=\frac{(e-c^2)z_1+c}{1-cz_1},\; v_2(z_2)=\frac{(f-d^2)z_2+d}{1-dz_2},
	\end{align*}
	where $a=u(0, 0),\; c=v_1(0),\; d=v_2(0),\; e=v^\prime_1(0),\; f=v^\prime_2(0)$.
\end{cor}
\begin{proof}[\bf Proof of Theorem \ref{thm-2.1}]
	First we suppose that $W_{u, v}$ is  $\mathcal{J}$-symmetric. Then, we have 
	\begin{align}\label{Eq-2.1}
		W_{u, v}\mathcal{J}K_w(z)=\mathcal{J}W^*_{u, v}K_w(z)
	\end{align}
for all $z=(z_1, z_2, \ldots, z_n),w=(w_1, w_2, \ldots, w_n)\in\mathbb{D}^n$. An easy computation shows that 
\begin{align}\label{Eq-2.2}
W_{u, v}\mathcal{J}K_w(z)=u(z)\mathcal{J}K_w(v(z))=u(z_1,z_2, \ldots,z_n)\prod_{i=1}^{n}\frac{1}{(1-w_iv_i(z_i))}	
\end{align}
and
\begin{align}\label{Eq-2.3}
	\mathcal{J}W^*_{u, v}K_w(z)=\mathcal{J}\overline{u(w)}K_{v(w)}(z)=u(w)\overline{K_{v(w)}}(\bar{z})=u(w_1, w_2, \ldots, w_n)\prod_{i=1}^{n}\frac{1}{(1-z_iv_i(w_i))}.
\end{align}
With the help of equations \eqref{Eq-2.2} and \eqref{Eq-2.3}, we obtain from \eqref{Eq-2.1} that
\begin{align}\label{Eq-2.4}
	u(z_1,z_2, \ldots,z_n)\prod_{i=1}^{n}\frac{1}{(1-w_iv_i(z_i))}=u(w_1, w_2, \ldots, w_n)\prod_{i=1}^{n}\frac{1}{(1-z_iv_i(w_i))}
\end{align}
By substituting $ w_1=w_2=\ldots=w_n=0$ in \eqref{Eq-2.4}, we obtain
\begin{align}\label{Eq-2.5}
	u(z_1,z_2, \ldots,z_n)=u(0, 0, \ldots, 0)\prod_{i=1}^{n}\frac{1}{(1-v_i(0)z_i)}.
\end{align}
Since $u$ is not identically zero on $\mathbb{D}^n$, we have that $u(0, 0, \ldots, 0)\neq0$. By the virtue of equations \eqref{Eq-2.4} and \eqref{Eq-2.5}, it follows that
\begin{align}\label{Eq-2.6}
	\prod_{i=1}^{n}(1-v_i(0)z_i)\times\prod_{i=1}^{n}(1-w_iv_i(z_i))=	\prod_{i=1}^{n}(1-v_i(0)w_i)\times\prod_{i=1}^{n}(1-z_iv_i(w_i)).
\end{align}
By differentiating \eqref{Eq-2.6} partially with respect to $w_j$ (where $1\leq j\leq n$), we obtain
\begin{align}\label{Eq-2.7}
		\prod_{i=1}^{n}&(1-v_i(0)z_i)\times\prod_{i\neq j}^{n}(1-w_iv_i(z_i))\times(-v_j(z_j))\\=&\prod_{i\neq j}^{n}(1-v_i(0)w_i)\times\prod_{i=1}^{n}(1-z_iv_i(w_i))\times(-v_j(0))\nonumber\\&+\prod_{i=1}^{n}(1-v_i(0)w_i)\times\prod_{i\neq j}^{n}(1-z_iv_i(w_i))\times(-z_jv^\prime_j(w_j))\nonumber.
\end{align} 
Letting $z_i=0$ for $i\neq j$ and $w_i=0$ for all $ i=1, 2, \ldots, n$ in \eqref{Eq-2.7}, we obtain
\begin{align*}
	(1-v_j(0)z_j)v_j(z_j)=z_jv^\prime_j(0)+(1-v_j(0)z_j)v_j(0).
\end{align*}
This implies that 
\begin{align*}
	v_j(z_j)=\frac{(v^\prime_j(0)-v^2_j(0))z_j+v_j(0)}{	1-v_j(0)z_j}.
\end{align*}
Thus, we obtain
\begin{align*}
	u(z_1, z_2, \ldots, z_n)=a\prod_{i=1}^{n}\frac{1}{(1-c_iz_i)}\;\mbox{and}\; v_j(z_j)=\frac{(d_j-c_j^2)z_j+c_j}{1-c_jz_j}\; \mbox{for}\; j=1, 2, \ldots, n
\end{align*}
with $ a=u(0, 0, \ldots, 0) $ and $c_j=v_j(0)$, $d_j=v^\prime(0)$ for $j=1, 2, \ldots, n $.\vspace{2mm}

Conversely, we suppose that 
\begin{align}\label{Eq-2.8}
	u(z_1, z_2, \ldots, z_n)=a\prod_{i=1}^{n}\frac{1}{(1-c_iz_i)}\;\mbox{and}\; v_j(z_j)=\frac{(d_j-c_j^2)z_j+c_j}{1-c_jz_j}\; \mbox{for}\; j=1, 2, \ldots, n
\end{align}
with $ a=u(0, 0, \ldots, 0) $ and $c_j=v_j(0)$, $d_j=v^\prime(0)$ for $j=1, 2, \ldots, n $. By using \eqref{Eq-2.8}, we have from equations \eqref{Eq-2.2} and \eqref{Eq-2.3} that
\begin{align}\label{Eq-2.9}
	W_{u, v}\mathcal{J}K_w(z)=&a\prod_{i=1}^{n}\frac{1}{(1-c_iz_i)}\times\prod_{i=1}^{n}\frac{1}{\bigg(1-w_i\bigg(\frac{(d_i-c_i^2)z_i+c_i}{1-c_iz_i}\bigg)\bigg)}\\=&a\prod_{i=1}^{n}\frac{1}{(1-c_iz_i-w_i((d_i-c_i^2)z_i+c_i))}\nonumber
\end{align}
and 
\begin{align}\label{Eq-2.10}
\mathcal{J}W^*_{u, v}K_w(z)=&a\prod_{i=1}^{n}\frac{1}{(1-c_iw_i)}\times\prod_{i=1}^{n}\frac{1}{\bigg(1-z_i\bigg(\frac{(d_i-c_i^2)w_i+c_i}{1-c_iw_i}\bigg)\bigg)}\\=&a\prod_{i=1}^{n}\frac{1}{(1-c_iw_i-z_i((d_i-c_i^2)w_i+c_i))}\nonumber.	
\end{align}
Since
\begin{align*}
	1-c_iz_i-w_i((d_i-c_i^2)z_i+c_i)=1-c_iw_i-z_i((d_i-c_i^2)w_i+c_i),
\end{align*} we obtain from equations \eqref{Eq-2.9} and \eqref{Eq-2.10} that $	W_{u, v}\mathcal{J}K_w(z)=\mathcal{J}W^*_{u, v}K_w(z)$, which implies that $W_{u, v}$ is $\mathcal{J}$-symmetric. This completes the proof.
\end{proof}
The following result centers on the investigation of which combinations of $u$ and $v$ create complex symmetric weighted composition operators $W_{u, v, k_1, k_2, \ldots, k_n}$ with respect to the standard conjugation $\mathcal{J}$.

\begin{thm}\label{thm-2.2}
	Let $u:\mathbb{D}^n\to \mathbb{C} $ be a nonzero analytic function and, let $v:\mathbb{D}^n\to\mathbb{D}^n $ with $v(z_1, z_2, \ldots, z_n)=(v_1(z_1), v_2(z_2),\ldots, v_n(z_n))$ be an analytic self-map of $\mathbb{D}^n$ such that  $W_{u, v, k_1, k_2, \ldots, k_n}$ is bounded on $H^2(\mathbb{D}^n)$. Then $W_{u, v, k_1, k_2, \ldots, k_n}$ is $\mathcal{J}$-symmetric if, and only if,
	\begin{align*}
		u(z_1, z_2, \ldots, z_n)=a\prod_{i=1}^{n}\frac{z^{k_i}_i}{(1-c_iz_i)^{k_i+1}}\;\mbox{and}\; v_j(z_j)=\frac{(d_j-c_j^2)z_j+c_j}{1-c_jz_j}\; \mbox{for}\; j=1, 2, \ldots, n,
	\end{align*}
	where $ a=\frac{\partial^{m}}{\partial z^{k_1}_1\partial z^{k_2}_2 \ldots \partial z^{k_n}_n} u(0, 0, \ldots, 0) $ and $c_j=v_j(0)$, $d_j=v^\prime(0)$ ($j=1, 2, \ldots, n $).
\end{thm}
\begin{rem}
	First, we have to find conditions on $c_i$ and $d_i$ such that $v_i(\mathbb{D})\subset\mathbb{D}$.
	A simple computation yields that
	\begin{align*}
		|v_i(z_i)|&=\bigg|\frac{(d_i-c_i^2)z_i+c_i}{1-c_iz_i}\bigg|\\&=\bigg|c_i+\frac{d_iz_i}{1-c_iz_i}\bigg|\\&\ge|c_i|-\bigg|\frac{d_iz_i}{1-c_iz_i}\bigg|\\&>|c_i|+\bigg|\frac{d_iz_i}{c_iz_i}\bigg|\\&=|c_i|+\bigg|\frac{d_i}{c_i}\bigg|.
	\end{align*}
	Therefore, if $|v_i(z_i)|<1$, then we see that $|c_i|+\bigg|\frac{d_i}{c_i}\bigg|<1$, which shows that $|d_i|<1-|c_i|^2$. Thus, it is easy to see that $v_i(\mathbb{D})\subset\mathbb{D}$ if $|d_i|<1-|c_i|^2$.\vspace{1.2mm} 
	
	Next, we have to find the conditions on $c_i$ and $d_i$ such that $	W_{u, v, k_1, k_2, \ldots, k_n}$ is bounded. A straightforward computation shows that 
	\begin{align*}
	||W_{u, v, k_1, k_2, \ldots, k_n}K_w(z)||=&\bigg|\bigg|u(z_1, z_2, \ldots, z_n)\prod_{i=1}^{n}\frac{k_i!\overline{w_i}^{k_i}}{(1-\overline{w_i}v_i(z_i))^{k_i+1}}\bigg|\bigg|\\=&\bigg|\bigg|a\prod_{i=1}^{n}\frac{z^{k_i}_i}{(1-c_iz_i)^{k_i+1}}\times\prod_{i=1}^{n}\frac{k_i!\overline{w_i}^{k_i}}{\bigg(1-\overline{w_i}\bigg(\frac{(d_i-c_i^2)z_i+c_i}{1-c_iz_i}\bigg)\bigg)^{k_i+1}}\bigg|\bigg|\\=&\bigg|\bigg|a\prod_{i=1}^{n}\frac{k_i!(\overline{w_i}z_i)^{k_i}}{(1-c_iz_i-\overline{w_i}((d_i-c_i^2)z_i+c_i))^{k_i+1}}\bigg|\bigg|\\\le&|a|\prod_{i=1}^{n}\frac{k_i!}{|(1-c_iz_i-d_i\overline{w_i}-c_i^2z_i\overline{w_i}+c_i\overline{w_i})|^{k_i+1}}\\\le&|a|\prod_{i=1}^{n}\frac{k_i!}{(1-|c_iz_i|-|d_iw_i|-|c_i^2z_iw_i|-|c_iw_i|)^{k_i+1}}\\\le&|a|\prod_{i=1}^{n}\frac{k_i!}{|(1-c_iz_i-d_i\overline{w_i}-c_i^2z_i\overline{w_i}+c_i\overline{w_i})|^{k_i+1}}\\\le&|a|\prod_{i=1}^{n}\frac{k_i!}{(1-|c_i|-|d_i|-|c_i^2|-|c_i|)^{k_i+1}}\\=&|a|\prod_{i=1}^{n}\frac{k_i!}{(1-2|c_i|-|d_i|-|c_i|^2)^{k_i+1}}.
	\end{align*}
	It follows that $||W_{u, v, k_1, k_2, \ldots, k_n}K_w(z)||$ is finite if $1-2|c_i|-|d_i|-|c_i|^2>0$. This turns out that $|d_i|<1-2|c_i|-|c_i|^2<(1-|c_i|)^2$. Since the span of the reproducing kernel functions is dense in  $H^2(\mathbb{D}^n)$, the quantity $||W_{u, v, k_1, k_2, \ldots, k_n}f(z)||$ is finite for all $f\in H^2(\mathbb{D}^n)$ if $|d_i|<(1-|c_i|)^2$. Therefore, the operator $W_{u, v, k_1, k_2, \ldots, k_n}$ is bounded on $H^2(\mathbb{D}^n)$ if $|d_i|<(1-|c_i|)^2$.
\end{rem}
\begin{proof}[\bf Proof of Theorem \ref{thm-2.2}]
Let $W_{u, v, k_1, k_2, \ldots, k_n}$ be $\mathcal{J}$-symmetric. Then, we have
\begin{align}\label{Eq-2.11}
W_{u, v, k_1, k_2, \ldots, k_n}\mathcal{J}K_w(z)=\mathcal{J}W^*_{u, v, k_1, k_2, \ldots, k_n}K_w(z)
\end{align}
for all $z=(z_1, z_2, \ldots, z_n),w=(w_1, w_2, \ldots, w_n)\in\mathbb{D}^n$. Moreover, it can be easily seen through computation that
\begin{align}\label{Eq-2.12}
W_{u, v, k_1, k_2, \ldots, k_n}\mathcal{J}K_w(z)=&u(z)\bigg(\frac{\partial^{m}}{\partial z^{k_1}_1\partial z^{k_2}_2 \ldots \partial z^{k_n}_n}\prod_{i=1}^{n}\frac{1}{(1-w_iz_i)}\bigg)(v(z))\\=&u(z_1, z_2, \ldots, z_n)\prod_{i=1}^{n}\frac{k_i!w^{k_i}_i}{(1-w_iv_i(z_i))^{k_i+1}}\nonumber	
\end{align}
and
\begin{align}\label{Eq-2.13}
	\mathcal{J}W^*_{u, v, k_1, k_2, \ldots, k_n}K_w(z)=&\mathcal{J}\overline{u(w)}K^{[k_1, k_2, \ldots, k_n]}_{v(w)}(z)\\=&u(w)\overline{K^{[k_1, k_2, \ldots, k_n]}_{v(w)}(\bar{z})}\nonumber\\=&u(w_1, w_2, \ldots, w_n)\prod_{i=1}^{n}\frac{k_i!z^{k_i}_i}{(1-z_iv_i(w_i))^{k_i+1}}\nonumber.
\end{align}
 Thus, the equation \eqref{Eq-2.11} is equivalent to
 \begin{align}\label{Eq-2.14}
 	u(z_1, z_2, \ldots, z_n)\prod_{i=1}^{n}\frac{k_i!w^{k_i}_i}{(1-w_iv_i(z_i))^{k_i+1}}=u(w_1, w_2, \ldots, w_n)\prod_{i=1}^{n}\frac{k_i!z^{k_i}_i}{(1-z_iv_i(w_i))^{k_i+1}}.
 \end{align}
Substituting $z_1=0, z_2=0, \ldots, z_n=0$ in \eqref{Eq-2.14}, we obtain $u(0, z_2, \ldots, z_n)=0, u(z_1, 0, \ldots, z_n)=0, \ldots, u(z_1, z_2, \ldots, z_{n-1},0)=0$ and $u(0, 0, \ldots, 0)=0$ respectively, which leads us to the conclusion that $u$ contains factors of $z_1, z_2, \ldots, z_n$.\vspace{1.2mm} 

Let $u(z_1, z_2, \ldots, z_n)=h(z_1, z_2, \ldots, z_n)\prod_{i=1}^{n}z^{t_i}_i$, where $t_i(i=1, 2,\ldots,n)\in\mathbb{N}$ and $h$ is analytic with $h(0, z_2, \ldots, z_n)\neq0, h(z_1, 0, \ldots, z_n)\neq0, \ldots, h(z_1, z_2, \ldots, z_{n-1},0)\neq0$ and $h(0, 0, \ldots, 0)\neq0$. Then, we obtain from \eqref{Eq-2.14} that
\begin{align}\label{Eq-2.15}
		h(z_1, z_2, \ldots, z_n)\prod_{i=1}^{n}\frac{k_i!w^{k_i}_iz^{t_i}_i}{(1-w_iv_i(z_i))^{k_i+1}}=h(w_1, w_2, \ldots, w_n)\prod_{i=1}^{n}\frac{k_i!w^{t_i}_iz^{k_i}_i}{(1-z_iv_i(w_i))^{k_i+1}}.
\end{align}
Now, we claim that $k_i=t_i$ for all $i=1, 2, \ldots, n$. If $k_i>t_i$, it follows from \eqref{Eq-2.15} that 
\begin{align}\label{Eq-2.16}
	h(z_1, z_2, \ldots, z_n)\prod_{i=1}^{n}\frac{k_i!w^{k_i-t_i}_i}{(1-w_iv_i(z_i))^{k_i+1}}=h(w_1, w_2, \ldots, w_n)\prod_{i=1}^{n}\frac{k_i!z^{k_i-t_i}_i}{(1-z_iv_i(w_i))^{k_i+1}}
\end{align} 
Setting $z_1=0$ in \eqref{Eq-2.16}, we obtain  $h(0, z_2, \ldots, z_n)=0$, which is a contradiction because $h(0, z_2, \ldots, z_n)\neq0$. If  $k_i<t_i$, we obtain from \eqref{Eq-2.15} that
\begin{align}\label{Eq-2.17}
	h(z_1, z_2, \ldots, z_n)\prod_{i=1}^{n}\frac{k_i!z^{t_i-k_i}_i}{(1-w_iv_i(z_i))^{k_i+1}}=h(w_1, w_2, \ldots, w_n)\prod_{i=1}^{n}\frac{k_i!w^{t_i-k_i}_i}{(1-z_iv_i(w_i))^{k_i+1}}.	
\end{align}
Putting $w_1=0$ in \eqref{Eq-2.17}, we get $h(z_1, z_2, \ldots, z_n)\equiv0$ on $\mathbb{D}^n$, which is a contradiction. Therefore $k_i=t_i$ for all $i=1, 2, \ldots, n$ and hence the equation \eqref{Eq-2.15} reduces to 
\begin{align}\label{Eq-2.18}
		h(z_1, z_2, \ldots, z_n)\prod_{i=1}^{n}\frac{1}{(1-w_iv_i(z_i))^{k_i+1}}=h(w_1, w_2, \ldots, w_n)\prod_{i=1}^{n}\frac{1}{(1-z_iv_i(w_i))^{k_i+1}}.
\end{align}
Letting $w_i=0$ for all $i=1, 2, \ldots, n$ in \eqref{Eq-2.18}, we obtain
\begin{align*}
		h(z_1, z_2, \ldots, z_n)=h(0, 0, \ldots, 0)\prod_{i=1}^{n}\frac{1}{(1-v_i(0)z_i)^{k_i+1}}
\end{align*}
and hence
\begin{align}\label{Eq-2.19}
		u(z_1, z_2, \ldots, z_n)=h(0, 0, \ldots, 0)\prod_{i=1}^{n}\frac{z^{k_i}_i}{(1-v_i(0)z_i)^{k_i+1}}.
\end{align}
Since $h(0, 0, \ldots, 0)\neq0$, it follows from equations \eqref{Eq-2.14} and\eqref{Eq-2.19} that
\begin{align}\label{Eq-2.20}
	\prod_{i=1}^{n}(1-v_i(0)z_i)^{k_i+1}\times\prod_{i=1}^{n}(1-w_iv_i(z_i))^{k_i+1}=\prod_{i=1}^{n}(1-v_i(0)w_i)^{k_i+1}\times\prod_{i=1}^{n}(1-z_iv_i(w_i))^{k_i+1}.
\end{align}
Taking the partial derivative with respect to the variable $w_j(1\le j\le n)$ in \eqref{Eq-2.20}, we obtain 
\begin{align}\label{Eq-2.21}
(k_j&+1)\prod_{i=1}^{n}(1-v_i(0)z_i)^{k_i+1}\times\prod_{i\neq j}^{n}(1-w_iv_i(z_i))^{k_i+1}\times(1-w_jv_j(z_j))^{k_j}(-v_j(z_j))\\=&(k_j+1)\prod_{i\neq j}^{n}(1-v_i(0)w_i)^{k_i+1}\times\prod_{i=1}^{n}(1-z_iv_i(w_i))^{k_i+1}\times(1-v_j(0)w_j)^{k_j}(-v_j(0))\nonumber\\&+(k_j+1)\prod_{i=1}^{n}(1-v_i(0)w_i)^{k_i+1}\times\prod_{i\neq j}^{n}(1-z_iv_i(w_i))^{k_i+1}\times(1-z_jv_j(w_j))^{k_j}(-z_jv^\prime_j(w_j))\nonumber.
\end{align}
Letting $w_i=0$ for all $i=1, 2, \ldots,n$ and $z_i=0$ for $i\neq j$ in \eqref{Eq-2.21}, we obtain
\begin{align*}
(1-v_j(0)z_j)^{k_j+1}(v_j(z_j))=(1-v_j(0)z_j)^{k_j+1}(v_j(0))+(1-v_j(0)z_j)^{k_j}(z_jv^\prime_j(0)), 	
\end{align*}
which implies that 
\begin{align*}
v_j(z_j)=\frac{(v^\prime_j(0)-v^2_j(0))z_j+v_j(0)}{1-v_j(0)z_j}.
\end{align*}
Thus, we easily obtain the following
\begin{align*}
	u(z_1, z_2, \ldots, z_n)=a\prod_{i=1}^{n}\frac{z^{k_i}_i}{(1-c_iz_i)}\;\mbox{and}\; v_j(z_j)=\frac{(d_j-c_j^2)z_j+c_j}{1-c_jz_j}\;\; \mbox{for}\; j=1, 2, \ldots, n,
\end{align*}
where $ a=\frac{\partial^{m}}{\partial z^{k_1}_1\partial z^{k_2}_2 \ldots \partial z^{k_n}_n} u(0, 0, \ldots, 0) $ and $c_j=v_j(0)$, $d_j=v^\prime(0)$( $j=1, 2, \ldots, n $).
\vspace{1.2mm}
 Conversely, suppose that
\begin{align*}
	u(z_1, z_2, \ldots, z_n)=a\prod_{i=1}^{n}\frac{z^{k_i}_i}{(1-c_iz_i)^{k_i+1}}\;\mbox{and}\; v_j(z_j)=\frac{(d_j-c_j^2)z_j+c_j}{1-c_jz_j}\; \mbox{for}\; j=1, 2, \ldots, n,
\end{align*}
where
\begin{align*}
	a=\frac{\partial^{m}}{\partial z^{k_1}_1\partial z^{k_2}_2 \ldots \partial z^{k_n}_n} u(0, 0, \ldots, 0)\;\; \mbox{and}\; c_j=v_j(0),\; d_j=v^\prime(0)\;\; \mbox{for}\;\; j=1, 2, \ldots, n ).
\end{align*}
 In view of \eqref{Eq-2.12} and \eqref{Eq-2.13}, a simple computation shows that
\begin{align}\label{Eq-2.22}
	W_{u, v, k_1, k_2, \ldots, k_n}\mathcal{J}K_w(z)=&u(z_1, z_2, \ldots, z_n)\prod_{i=1}^{n}\frac{k_i!w^{k_i}_i}{(1-w_iv_i(z_i))^{k_i+1}}\\=&a\prod_{i=1}^{n}\frac{z^{k_i}_i}{(1-c_iz_i)^{k_i+1}}\times\prod_{i=1}^{n}\frac{k_i!w^{k_i}_i}{\bigg(1-w_i\bigg(\frac{(d_i-c_i^2)z_i+c_i}{1-c_iz_i}\bigg)\bigg)^{k_i+1}}\nonumber\\=&a\prod_{i=1}^{n}\frac{k_i!(w_iz_i)^k_i}{(1-c_iz_i-w_i((d_i-c_i^2)z_i+c_i))^{k_i+1}}\nonumber
\end{align}
and
\begin{align}\label{Eq-2.23}
	\mathcal{J}W^*_{u, v, k_1, k_2, \ldots, k_n}K_w(z)=&u(w_1, w_2, \ldots, w_n)\prod_{i=1}^{n}\frac{k_i!z^{k_i}_i}{(1-z_iv_i(w_i))^{k_i+1}}\\=&a\prod_{i=1}^{n}\frac{w^{k_i}_i}{(1-c_iw_i)^{k_i+1}}\times\prod_{i=1}^{n}\frac{k_i!z^{k_i}_i}{\bigg(1-z_i\bigg(\frac{(d_i-c_i^2)w_i+c_i}{1-c_iw_i}\bigg)\bigg)^{k_i+1}}\nonumber\\=&a\prod_{i=1}^{n}\frac{k_i!(w_iz_i)^k_i}{(1-c_iw_i-z_i((d_i-c_i^2)w_i+c_i))^{k_i+1}}\nonumber
\end{align}
for any $z=(z_1, z_2, \ldots, z_n),w=(w_1, w_2, \ldots, w_n)\in\mathbb{D}^n$. Since 
\begin{align*}
1-c_iz_i-w_i((d_i-c_i^2)z_i+c_i)=1-c_iw_i-z_i((d_i-c_i^2)w_i+c_i),	
\end{align*}
we obtain from \eqref{Eq-2.22} and \eqref{Eq-2.23} that $	W_{u, v, k_1, k_2, \ldots, k_n}\mathcal{J}K_w(z)=	\mathcal{J}W^*_{u, v, k_1, k_2, \ldots, k_n}K_w(z)$, which implies that  $W_{u, v, k_1, k_2, \ldots, k_n}$ is $\mathcal{J}$-symmetric. This completes the proof.
\end{proof}
\section{Hermitian property of composition-differentiation operators on the Hardy space $H^2(\mathbb{D}^n)$ }
A bounded linear operator $T$ is Hermitian (self-adjoint) if $T^*=T$. Recent research has been focused on investigating the self-adjointness of weighted composition-differentiation operators in the reproducing kernel Hilbert space (see \cite{Fatehi-Moradi-2021, Lim-Khoi-JMAA-2018, Lo-Loh-JMAA-2023} and other related works). In 2021, Fatehi and Hummond \cite{Fatehi-Hammond-CAOT-2021} presented a complete characterization of the self-adjointness of weighted composition-differentiation operators in the Hardy space over the unit disk $\mathbb{D}$. In 2023, Liu \emph{et~al} \cite{Liu-Ponnusamy-Xie-LMA-2023} discussed the self-adjointness of the weighted composition-differentiation operator in the Bergman space $\mathcal{A}^2_{\alpha}(\mathbb{D})$. In 2012, Le \cite{Le-JMAA-2012} derived characterizations of self-adjoint weighted composition operators on Hilbert spaces of analytic functions over the unit ball $\mathbb{B}_n:=\{z\in\mathbb{C}^n:|z|<1\}$. On a similar note, in \cite{Han-Wang-JMAA-2019}, Han and Wang investigated the self-adjointness of the composition operator on Hilbert spaces of analytic functions over the same unit ball $\mathbb{B}_n$. No research has been carried out on the self-adjointness of the composition-differentiation operator in relation to the Hardy spaces over the polydisk $\mathbb{D}^n$.\vspace{1,2mm} 

Our aim is to investigate the self-adjoint properties over the polydisk $\mathbb{D}^n$. In the following result, we establish a condition that is both necessary and sufficient for the bounded composition-differentiation operator $W_{u, v, k_1, k_2, \ldots, k_n}$ to satisfy Hermitian properties on $H^2(\mathbb{D}^n)$.
\begin{thm}\label{thm-2.3}
Let $u:\mathbb{D}^n\to \mathbb{C} $ be a nonzero analytic function and, let $v:\mathbb{D}^n\to\mathbb{D}^n $ with $v(z_1, z_2, \ldots, z_n)=(v_1(z_1), v_2(z_2),\ldots, v_n(z_n))$ be an analytic self-map of $\mathbb{D}^n$ such that  $W_{u, v, k_1, k_2, \ldots, k_n}$ is bounded on $H^2(\mathbb{D}^n)$. Then $W_{u, v, k_1, k_2, \ldots, k_n}$ is Hermitian if, and only if,
\begin{align*}
	u(z_1, z_2, \ldots, z_n)=a\prod_{i=1}^{n}\frac{z^{k_i}_i}{(1-\bar{c_i}z_i)^{k_i+1}}\;\mbox{and}\; v_j(z_j)=\frac{(d_j-|c_j|^2)z_j+c_j}{1-\bar{c_j}z_j}\; \mbox{for}\; j=1, 2, \ldots, n,
\end{align*}
where $a,d_j(j=1, 2, \ldots, n)\in\mathbb{R}$ and $c_j(j=1, 2, \ldots, n)\in\mathbb{C}$.
\end{thm}
The defining formulas of $W_{u, v}$ and $W_{u, v, k_1, k_2, \ldots, k_n}$ make it evident that $W_{u, v, 0, 0, \ldots, 0} = W_{u, v}$. Moreover, the implications of Theorem \ref{thm-2.3} reveal that when $k_1=k_2=\ldots=k_n=0$, $W_{u, v, k_1, k_2, \ldots, k_n}$ coincides with $W_{u, v}$. Consequently, we can readily derive the following result from Theorem \ref{thm-2.3}.

\begin{cor}
	Let $u:\mathbb{D}^n\to \mathbb{C} $ be a nonzero analytic function and, let $v:\mathbb{D}^n\to\mathbb{D}^n $ with $v(z_1, z_2, \ldots, z_n)=(v_1(z_1), v_2(z_2),\ldots, v_n(z_n))$ be an analytic self-map of $\mathbb{D}^n$ such that  $W_{u, v}$ is bounded on $H^2(\mathbb{D}^n)$. Then $W_{u, v}$ is Hermitian if, and only if,
	\begin{align*}
		u(z_1, z_2, \ldots, z_n)=a\prod_{i=1}^{n}\frac{1}{(1-\bar{c_i}z_i)^{k_i+1}}\;\mbox{and}\; v_j(z_j)=\frac{(d_j-|c_j|^2)z_j+c_j}{1-\bar{c_j}z_j}\; \mbox{for}\; j=1, 2, \ldots, n,
	\end{align*}
	where $a,d_j(j=1, 2, \ldots, n)\in\mathbb{R}$ and $c_j(j=1, 2, \ldots, n)\in\mathbb{C}$.

\end{cor}
\begin{proof}[\bf Proof of Theorem \ref{thm-2.3}]
	Let  $W_{u, v, k_1, k_2, \ldots, k_n}$ be Hermitian. Then $W^*_{u, v, k_1, k_2, \ldots, k_n}=W_{u, v, k_1, k_2, \ldots, k_n}$, which is equivalent to 
	\begin{align}\label{Eq-2.24}
		W^*_{u, v, k_1, k_2, \ldots, k_n}K_w(z)=W_{u, v, k_1, k_2, \ldots, k_n}K_w(z)
	\end{align}
 for all $w, z\in\mathbb{D}^n$. A simple computation shows that
 \begin{align}\label{Eq-2.25}
 		W_{u, v, k_1, k_2, \ldots, k_n}K_w(z)=&u(z)\bigg(\frac{\partial^{m}}{\partial z^{k_1}_1\partial z^{k_2}_2 \ldots \partial z^{k_n}_n}\prod_{i=1}^{n}\frac{1}{(1-\bar{w_i}z_i)}\bigg)(v(z))\\=&u(z_1, z_2, \ldots, z_n)\prod_{i=1}^{n}\frac{k_i!\bar{w_i}^{k_i}}{(1-\bar{w_i}v_i(z_i))^{k_i+1}}\nonumber
 \end{align}
and 
\begin{align}\label{Eq-2.26}
W^*_{u, v, k_1, k_2, \ldots, k_n}K_w(z)=&\overline{u(w)}K^{[k_1, k_2, \ldots, k_n]}_{v(w)}(z)\\=&\overline{u(w_1, w_2, \ldots, w_n)}\prod_{i=1}^{n}\frac{k_i!z^{k_i}_i}{(1-z_i\overline{v_i(w_i)})^{k_i+1}}\nonumber.	
\end{align} 
By using the equations \eqref{Eq-2.25} and \eqref{Eq-2.26}, we obtain from \eqref{Eq-2.24} that  
\begin{align}\label{Eq-2.27}
	u(z_1, z_2, \ldots, z_n)\prod_{i=1}^{n}\frac{k_i!\bar{w_i}^{k_i}}{(1-\bar{w_i}v_i(z_i))^{k_i+1}}=\overline{u(w_1, w_2, \ldots, w_n)}\prod_{i=1}^{n}\frac{k_i!z^{k_i}_i}{(1-z_i\overline{v_i(w_i)})^{k_i+1}}.
\end{align}
By applying a similar argument as in the proof of Theorem \ref{thm-2.2}, we can assume that $u(z_1, z_2, \ldots, z_n)=h(z_1, z_2, \ldots, z_n)\prod_{i=1}^{n}z^{k_i}_i$ and $h$ is analytic with $h(0, z_2, \ldots, z_n)\neq0, h(z_1, 0, \ldots, z_n)\neq0, \ldots, h(z_1, z_2, \ldots, z_{n-1},0)\neq0$ and $h(0, 0, \ldots, 0)\neq0$. Then, we obtain from \eqref{Eq-2.27} that
\begin{align}\label{Eq-2.28}
		h(z_1, z_2, \ldots, z_n)\prod_{i=1}^{n}\frac{1}{(1-\bar{w_i}v_i(z_i))^{k_i+1}}=\overline{h(w_1, w_2, \ldots, w_n)}\prod_{i=1}^{n}\frac{1}{(1-z_i\overline{v_i(w_i)})^{k_i+1}}.
\end{align}
Putting $w_i=0$ for all $i=1, 2, \ldots,n$ in \eqref{Eq-2.28}, we obtain 
\begin{align*}
	h(z_1, z_2, \ldots, z_n)=\overline{h(0, 0, \ldots, 0)}\prod_{i=1}^{n}\frac{1}{(1-\overline{v_i(0)}z_i)^{k_i+1}}
\end{align*}
and hence 
\begin{align*}
	u(z_1, z_2, \ldots, z_n)=a\prod_{i=1}^{n}\frac{z^{k_i}_i}{(1-\bar{c_i}z_i)^{k_i+1}}
\end{align*}
with $a=\overline{h(0, 0, \ldots, 0)}$ and $c_i=v_i(0)$ for $i=1, 2, \ldots, n$. Substituting this in \eqref{Eq-2.27}, we obtain
 \begin{align}\label{Eq-2.29}
 	\bar{a}\prod_{i=1}^{n}(1-\bar{c_i}z_i)^{k_i+1}\times\prod_{i=1}^{n}(1-\bar{w_i}v_i(z_i))^{k_i+1}=a\prod_{i=1}^{n}(1-c_i\bar{w_i})^{k_i+1}\times\prod_{i=1}^{n}(1-z_i\overline{v_i(w_i)})^{k_i+1}
 \end{align}
for any $z=(z_1, z_2, \ldots, z_n),w=(w_1, w_2, \ldots, w_n)\in\mathbb{D}^n$. If we set $w_1=w_2= \ldots= w_n=0$, then it is easy to see that $\bar{a}=a$, and hence $a\in\mathbb{R}$. Taking the derivative partially with respect to $\bar{w_j}(1\le j\le n)$ in \eqref{Eq-2.29}, we obtain
\begin{align}\label{Eq-2.30}
	(k_j+1)&\prod_{i=1}^{n}(1-\bar{c_i}z_i)^{k_i+1}\times\prod_{i\neq j}^{n}(1-\bar{w_i}v_i(z_i))^{k_i+1}\times(1-\bar{w_j}v_j(z_j))^{k_j}(-v_j(z_j))\\=&(k_j+1)\prod_{i\neq j}^{n}(1-c_i\bar{w_i})^{k_i+1}\times\prod_{i=1}^{n}(1-z_i\overline{v_i(w_i)})^{k_i+1}\times(1-c_j\bar{w_j})^{k_j}(-c_j)\nonumber\\&+(k_j+1)\prod_{i=1}^{n}(1-c_i\bar{w_i})^{k_i+1}\times\prod_{i\neq j}^{n}(1-z_i\overline{v_i(w_i)})^{k_i+1}\times(1-z_j\overline{v_j(w_j)})^{k_j}(-z_j\overline{v^\prime_j(w_j)})\nonumber
\end{align}
 for any $z=(z_1, z_2, \ldots, z_n),w=(w_1, w_2, \ldots, w_n)\in\mathbb{D}^n$. Setting $w_1=w_2= \cdots= w_n=0$ in \eqref{Eq-2.30}, we have 
 \begin{align*}
 	(1-\bar{c_j}z_j)^{k_j+1}v_j(z_j)=	(1-\bar{c_j}z_j)^{k_j+1}c_j+	(1-\bar{c_j}z_j)^{k_j}(z_j\overline{v^\prime_j(0)})
 \end{align*}
which implies that 
\begin{align}\label{Eq-2.31}
	v_j(z_j)=\frac{(1-\bar{c_j}z_j)c_j+z_j\overline{v^\prime_j(0)}}{1-\bar{c_j}z_j}=\frac{(d_j-|c_j|^2)z_j+c_j}{1-\bar{c_j}z_j} \; \mbox{for all} \; j=1, 2, \ldots, n
\end{align}
with $ d_j=\overline{v^\prime_j(0)}$ and $c_j=v_j(0)$.\vspace{2mm}

 Differentiating \eqref{Eq-2.31} partially with respect to $z_j$, we obtain $v^\prime_j(0)=d_j$. Therefore, we see that $\bar{d_j}=d_j$ which implies that $d_j(1\le j\le n)$ are real numbers. Conversely, assume that
\begin{align*}
	u(z_1, z_2, \ldots, z_n)=a\prod_{i=1}^{n}\frac{z^{k_i}_i}{(1-\bar{c_i}z_i)^{k_i+1}}\;\mbox{and}\; v_j(z_j)=\frac{(d_j-|c_j|^2)z_j+c_j}{1-\bar{c_j}z_j}\; \mbox{for}\; j=1, 2, \ldots, n,
\end{align*}
where $a,d_j(j=1, 2, \ldots, n)\in\mathbb{R}$ and $c_j(j=1, 2, \ldots, n)\in\mathbb{C}$. In view of \eqref{Eq-2.25} and \eqref{Eq-2.26}, we have
\begin{align}\label{Eq-2.32}
	W_{u, v, k_1, k_2, \ldots, k_n}K_w(z)=&u(z_1, z_2, \ldots, z_n)\prod_{i=1}^{n}\frac{k_i!\bar{w_i}^{k_i}}{(1-\bar{w_i}v_i(z_i))^{k_i+1}}\\=&a\prod_{i=1}^{n}\frac{z^{k_i}_i}{(1-\bar{c_i}z_i)^{k_i+1}}\times\prod_{i=1}^{n}\frac{k_i!\bar{w_i}^{k_i}}{\bigg(1-\bar{w_i}\bigg(\frac{(d_i-|c_i|^2)z_i+c_i}{1-\bar{c_i}z_i}\bigg)\bigg)^{k_i+1}}\nonumber\\=&a\prod_{i=1}^{n}\frac{k_i!(\bar{w_i}z_i)^{k_i}}{(1-\bar{c_i}z_i-\bar{w_i}((d_i-|c_i|^2)z_i+c_i))^{k_i+1}}\nonumber
\end{align}  
and
\begin{align}\label{Eq-2.33}
	W^*_{u, v, k_1, k_2, \ldots, k_n}K_w(z)=&\overline{u(w_1, w_2, \ldots, w_n)}\prod_{i=1}^{n}\frac{k_i!z^{k_i}_i}{(1-z_i\overline{v_i(w_i)})^{k_i+1}}\\=&\overline{a\prod_{i=1}^{n}\frac{w^{k_i}_i}{(1-\bar{c_i}w_i)^{k_i+1}}}\times\prod_{i=1}^{n}\frac{k_i!z^{k_i}_i}{\bigg(1-z_i(\overline{\bigg(\frac{(d_i-|c_i|^2)w_i+c_i}{1-\bar{c_i}w_i}\bigg)}\bigg)^{k_i+1}}\nonumber\\=&\bar{a}\prod_{i=1}^{n}\frac{k_i!(\bar{w_i}z_i)^{k_i}}{(1-c_i\bar{w_i}-z_i((\bar{d_i}-|c_i|^2)\bar{w_i}+\bar{c_i}))^{k_i+1}}\nonumber
\end{align}
 for any $z=(z_1, z_2, \ldots, z_n),w=(w_1, w_2, \ldots, w_n)\in\mathbb{D}^n$. Since
 \begin{align*}
 	1-\bar{c_i}z_i-\bar{w_i}((d_i-|c_i|^2)z_i+c_i)=&1-\bar{c_i}z_i-\bar{w_i}(d_i-|c_i|^2)z_i-c_i\bar{w_i}\\=&1-c_i\bar{w_i}-z_i((d_i-|c_i|^2)\bar{w_i}+\bar{c_i})
 \end{align*}
and $a,d_j(j=1, 2, \ldots, n)\in\mathbb{R}$, it follows from \eqref{Eq-2.32} and \eqref{Eq-2.33} that 
\begin{align*}
	W^*_{u, v, k_1, k_2, \ldots, k_n}K_w(z)=W_{u, v, k_1, k_2, \ldots, k_n}K_w(z),
\end{align*} 
which implies that $	W_{u, v, k_1, k_2, \ldots, k_n}$ is Hermitian on $H^2(\mathbb{D}^n)$. This completes the proof.
\end{proof}
\section{Sufficient condition for the normality of composition-differentiation operators on the Hardy space $H^2(\mathbb{D}^n)$}
Normal operators are those that commute with their own adjoint. In other words, a bounded linear operator T  is classified as normal if, and only if, $T^*T=TT^*$. In 2023, Liu \emph{et~al} \cite{Liu-Ponnusamy-Xie-LMA-2023} investigated the normality of composition-differentiation operators on the Bergman space, presenting necessary and sufficient conditions in their study. Lo and Loh \cite{Lo-Loh-JMAA-2023} investigated the class of normal generalized weighted composition operators on $H_{\gamma}$ and their relevance to the property of complex symmetry. In the last few years, substantial research has been dedicated to investigate the normality of weighted composition-differentiation operators in the reproducing kernel Hilbert space (see \cite{Fatehi-Moradi-2021, Fatehi-Hammond-CAOT-2021, Jun-Kim-Ko-Lee-JFA-2014, Bourdon-Narayan-JMAA-2010} and relevant references). Le \cite{Le-JMAA-2012} has conducted a study on the normality of weighted composition operators on Hilbert spaces of analytic functions over the unit ball $\mathbb{B}_n$. Similarly, Han and Wang \cite{Han-Wang-JMAA-2019} have explored the normality of the composition operator on Hilbert spaces of analytic functions over the same unit ball $\mathbb{B}_n$.\vspace{1.2mm}

 Until now, the entirety of normality research has been restricted to the unit disk and the unit ball, and it seems that no investigations have been carried on the polydisk. Our main aim is to fill this gap by exploring the normality properties over the polydisk and we obtain Theorem \ref{thm-2.4}. However, the significance of Theorem \ref{thm-2.4} lies in its limited scope, as it does not cover all normal weighted composition-differentiation operators on $H^2(\mathbb{D}^n)$. Therefore, we establish the precise conditions for an operator $W_{u, v, k_1, k_2, \ldots, k_n}$ to be normal, assuming that $u$ and $v_j(z_j)$ (where $j=1, 2, \ldots, n$) adopt a specific form that generalizes the symbols explained in Theorem \ref{thm-2.4}. Despite our efforts, we have not yet determined whether an operator $W_{u, v, k_1, k_2, \ldots, k_n}$, with $v_j(0)(0)\neq0$ for $j=1, 2, \ldots, n$, can be normal for different choices  of $u$ and $v_j(z_j)$.
We now establish a sufficient condition that guarantees the bounded operator $W_{u, v, k_1, k_2, \ldots, k_n}$ to be normal, given the complex symmetry under the standard conjugation $\mathcal{J}$.
\begin{thm}\label{thm-2.4}
	Let $u:\mathbb{D}^n\to \mathbb{C} $ be a nonzero analytic function and, let $v:\mathbb{D}^n\to\mathbb{D}^n $ with $v(z_1, z_2, \ldots, z_n)=(v_1(z_1), v_2(z_2),\ldots, v_n(z_n))$ be an analytic self-map of $\mathbb{D}^n$ such that  $W_{u, v, k_1, k_2, \ldots, k_n}$ is bounded on $H^2(\mathbb{D}^n)$. Suppose $W_{u, v, k_1, k_2, \ldots, k_n}$ is $\mathcal{J}$-symmetric and $v_j(0)=0$ for all $j=1, 2, \ldots, n$. Then $W_{u, v, k_1, k_2, \ldots, k_n}$ is normal.
\end{thm}
\begin{proof}[\bf Proof of Theorem \ref{thm-2.4}]
Since $W_{u, v, k_1, k_2, \ldots, k_n}$ is $\mathcal{J}$-symmetric and $v_j(0)=0$ for all $j=1, 2, \ldots, n$, it follows from Theorem \ref{thm-2.2} that  
\begin{align*}
	u(z_1, z_2, \ldots, z_n)=a\prod_{i=1}^{n}z^{k_i}_i\;\mbox{and}\; v_j(z_j)=d_jz_j\; \mbox{for}\; j=1, 2, \ldots, n,
\end{align*}
where $a, d_j(j=1, 2, \ldots, n)\in\mathbb{C}$. Then for any non-negative integers $m_1, m_2, \ldots, m_n$, we have
\begin{align}\label{Eq-2.34}
	||&W_{u, v, k_1, k_2, \ldots, k_n}e_{m_1, m_2, \ldots, m_n}||^2\\=&\sum_{j_1, j_2, \ldots, j_n=0}^{\infty}|\langle W_{u, v, k_1, k_2, \ldots, k_n}e_{m_1, m_2, \ldots, m_n}, e_{j_1, j_2, \ldots, j_n}\rangle|^2\nonumber\\=&\sum_{j_1, j_2, \ldots, j_n=0}^{\infty}\bigg|\bigg\langle u\bigg(\frac{\partial^{m}}{\partial z^{k_1}_1\partial z^{k_2}_2 \ldots \partial z^{k_n}_n}\prod_{i}^{n}z^{m_i}_i\bigg)(v), e_{j_1, j_2, \ldots, j_n}\bigg\rangle\bigg|^2\nonumber\\=&\sum_{j_1, j_2, \ldots, j_n=0}^{\infty}\bigg|\bigg\langle a\prod_{i=1}^{n}z^{k_i}_i\times\bigg(\prod_{i}^{n}\frac{m_i!}{(m_i-k_i)!}(d_iz_i)^{m_i-k_i}\bigg), e_{j_1, j_2, \ldots, j_n}\bigg\rangle\bigg|^2\nonumber\\=&\sum_{j_1, j_2, \ldots, j_n=0}^{\infty}\bigg|\bigg\langle a\prod_{i}^{n}\frac{m_i!d^{m_i-k_i}_i}{(m_i-k_i)!}\times\prod_{i=1}^{n}z^{m_i}_i, e_{j_1, j_2, \ldots, j_n}\bigg\rangle\bigg|^2\nonumber\\=&\sum_{j_1, j_2, \ldots, j_n=0}^{\infty}\bigg| a\prod_{i}^{n}\frac{m_i!d^{m_i-k_i}_i}{(m_i-k_i)!}\bigg|^2\bigg|\bigg\langle e_{m_1, m_2, \ldots, m_n}, e_{j_1, j_2, \ldots, j_n}\bigg\rangle\bigg|^2\nonumber\\=&\bigg| a\prod_{i}^{n}\frac{m_i!d^{m_i-k_i}_i}{(m_i-k_i)!}\bigg|^2\nonumber
\end{align}
 and
\begin{align}\label{Eq-2.35}
	||&W^*_{u, v, k_1, k_2, \ldots, k_n}e_{m_1, m_2, \ldots, m_n}||^2\\=&\sum_{j_1, j_2, \ldots, j_n=0}^{\infty}|\langle W^*_{u, v, k_1, k_2, \ldots, k_n}e_{m_1, m_2, \ldots, m_n}, e_{j_1, j_2, \ldots, j_n}\rangle|^2\nonumber\\=&\sum_{j_1, j_2, \ldots, j_n=0}^{\infty}|\langle e_{m_1, m_2, \ldots, m_n},  W_{u, v, k_1, k_2, \ldots, k_n}e_{j_1, j_2, \ldots, j_n}\rangle|^2\nonumber\\=&\sum_{j_1, j_2, \ldots, j_n=0}^{\infty}\bigg|\bigg\langle e_{m_1, m_2, \ldots, m_n}, u\bigg(\frac{\partial^{m}}{\partial z^{k_1}_1\partial z^{k_2}_2 \ldots \partial z^{k_n}_n}\prod_{i}^{n}z^{j_i}_i\bigg)(v)\bigg\rangle\bigg|^2\nonumber\\=&\sum_{j_1, j_2, \ldots, j_n=0}^{\infty}\bigg|\bigg\langle e_{m_1, m_2, \ldots, m_n}, a\prod_{i=1}^{n}z^{k_i}_i\times\bigg(\prod_{i}^{n}\frac{j_i!}{(j_i-k_i)!}(d_iz_i)^{j_i-k_i}\bigg)\bigg\rangle\bigg|^2\nonumber\\=&\sum_{j_1, j_2, \ldots, j_n=0}^{\infty}\bigg|\bigg\langle e_{m_1, m_2, \ldots, m_n}, a\prod_{i}^{n}\frac{j_i!d^{j_i-k_i}_i}{(j_i-k_i)!}\times\prod_{i=1}^{n}z^{j_i}_i\bigg)\bigg\rangle\bigg|^2\nonumber\\=&\sum_{j_1, j_2, \ldots, j_n=0}^{\infty}\bigg| a\prod_{i}^{n}\frac{j_i!d^{j_i-k_i}_i}{(j_i-k_i)!}\bigg|^2\bigg|\bigg\langle e_{m_1, m_2, \ldots, m_n}, e_{j_1, j_2, \ldots, j_n}\bigg\rangle\bigg|^2\nonumber\\=&\bigg| a\prod_{i}^{n}\frac{m_i!d^{m_i-k_i}_i}{(m_i-k_i)!}\bigg|^2\nonumber
\end{align}
as the set $\{e_{m_1, m_2, \ldots, m_n}:m_1, m_2, \ldots,m_n\ge0\}$ forms an orthonormal basis for $H(\mathbb{D}^n)$. Thus, it follows from \eqref{Eq-2.34} and \eqref{Eq-2.35} that for any non-negative integer $m_1, m_2, \ldots, m_n$, we have 
\begin{align}\label{Eq-2.36}
	||W_{u, v, k_1, k_2, \ldots, k_n}e_{m_1, m_2, \ldots, m_n}||^2=\bigg| a\prod_{i}^{n}\frac{m_i!d^{m_i-k_i}_i}{(m_i-k_i)!}\bigg|^2=||W^*_{u, v, k_1, k_2, \ldots, k_n}e_{m_1, m_2, \ldots, m_n}||^2.
\end{align}
 Consequently,  \eqref{Eq-2.36} leads to the conclusion that 
 \begin{align*}
 	||W_{u, v, k_1, k_2, \ldots, k_n}f||=||W^*_{u, v, k_1, k_2, \ldots, k_n}f||
 \end{align*} for any $f\in H^2(\mathbb{D}^n)$, and hence $W_{u, v, k_1, k_2, \ldots, k_n}$ is a normal operator. This completes the proof.
\end{proof}
From Theorem \ref{thm-2.4}, it is clear that when $k_1=k_2=\ldots=k_n=0$, $W_{u, v, k_1, k_2, \ldots, k_n}$ coincides with $W_{u, v}$. Consequently, we can readily derive the following result from Theorem \ref{thm-2.4}.
\begin{cor}\label{cor-2.3}
		Let $u:\mathbb{D}^n\to \mathbb{C} $ be a nonzero analytic function and, let $v:\mathbb{D}^n\to\mathbb{D}^n $ with $v(z_1, z_2, \ldots, z_n)=(v_1(z_1), v_2(z_2),\ldots, v_n(z_n))$ be an analytic self-map of $\mathbb{D}^n$ such that  $W_{u, v}$ is bounded on $H^2(\mathbb{D}^n)$. Suppose $W_{u, v}$ is $\mathcal{J}$-symmetric and $v_j(0)=0$ for all $j=1, 2, \ldots, n$. Then $W_{u, v}$ is normal.
\end{cor}
\begin{proof}[\bf Proof of Corollary \ref{cor-2.3}]
	Since $W_{u, v}$ is $\mathcal{J}$-symmetric and $v_j(0)=0$ for all $j=1, 2, \ldots, n$, it follows from Theorem \ref{thm-2.1} that
	\begin{align*}
		u(z_1, z_2, \ldots, z_n)=a\;\mbox{and}\; v_j(z_j)=d_jz_j\; \mbox{for}\; j=1, 2, \ldots, n,
	\end{align*}
	where $a, d_j(j=1, 2, \ldots, n)\in\mathbb{C}$.\\\vspace{1.2mm}
	 Now, we can proceed with the proof in a manner analogous to the proof of Theorem \ref{thm-2.4}. By setting $k_i=0$ for all $i=1, 2, \ldots, n$ in equation \eqref{Eq-2.36}, we obtain
\begin{align*}
	||W_{u, v}e_{m_1, m_2, \ldots, m_n}||^2=\bigg| a\prod_{i}^{n}d^{m_i}_i\bigg|^2=||W^*_{u, v}e_{m_1, m_2, \ldots, m_n}||^2,
\end{align*}
which leads to the conclusion that $||W_{u, v}f||=||W^*_{u, v}f||$ for any $f\in H^2(\mathbb{D}^n)$, and hence $W_{u, v}$ is a normal operator. This completes the proof.
\end{proof}

\section{Kernel of generalized composition-differentiation operators on Hardy space $H^2(\mathbb{D}^n) $}
In the study of composition operators for several spaces of functions, one of the significant findings is to explore the kernel of the operators. In this section, we investigate the kernel of generalized composition-differentiation operators $W_{u, v, k_1, k_2, \ldots, k_n}$ on the Hardy space $H^2(\mathbb{D}^n)$. The kernel of an operator refers to the collection of vectors that map to the zero vector. To explore the kernel of generalized composition-differentiation operators on $H^2(\mathbb{D}^n)$, we must introduce the space containing all polynomials in $\mathbb{C}^n$ of degree $\leq m$.\vspace{1.2mm} 

Let $\mathcal{P}^n_m(\mathbb{C})$ be the space containing all polynomials in $\mathbb{C}^n$ of degree $\leq m$ over the complex plane $\mathbb{C}$. The explicit definition of this space is as follows:
\begin{align*}
	\mathcal{P}^n_m(\mathbb{C})=\bigg\{\sum_{a_1+a_2+\cdots+a_n\leq m} A_{a_1, a_2, \ldots, a_n} \prod_{i=1}^{n} z_i^{a_i} : A_{a_1, a_2, \ldots, a_n} \in \mathbb{C}\bigg\}.
\end{align*}

 In 2014, Jung \emph{et~al} \cite[Proposition 3.2]{Jun-Kim-Ko-Lee-JFA-2014} have proved that the kernel of the adjoint of a complex symmetric weighted composition operator on $ {H}^2(\mathbb{D}) $ is simply $\{0\}$. In 2021, Han and Wang \cite[Proposition 2.9]{Han-Wang-BJMA-2021} demonstrated that the kernel of the adjoint of a nonzero $ J $-symmetric weighted composition-differentiation operator on the Hardy Hilbert space $ {H}^2(\mathbb{D}) $ is the whole complex plane $ \mathbb{C} $.\vspace{1.2mm}
 
 It is natural to raise the following question.
 \begin{ques}\label{Q-5.1}
 Can we explore the kernel of $W_{u, v, k_1, k_2, \ldots, k_n}$ on $H^2(\mathbb{D}^n)$ explicitly?
 \end{ques}
 
 In this section, our primary aim is to answer the Question \ref{Q-5.1}. It has been studied extensively the kernel of the composition-differentiation operator in different spaces defined over the unit disk $\mathbb{D}$. We will explore the kernel of the composition-differentiation operator $W_{u, v, k_1, k_2, \ldots, k_n}$ on $H^2(\mathbb{D}^n)$ in the Hardy space $H^2(\mathbb{D}^n)$ over the polydisk $\mathbb{D}^n$.\vspace{1.2mm}

By establishing the following result, we will show that both the kernels of the generalized composition-differentiation operators $W_{u, v, k_1, k_2, \ldots, k_n}$ and their adjoint are equal to $\mathcal{P}^n_{m-1}(\mathbb{C})$. Proving the following result is greatly influenced by the role of the Identity theorem in several complex variables.
\begin{thm}\label{thm-5.1}
	Let $u:\mathbb{D}^n\to \mathbb{C} $ be a nonzero analytic function and, let $v:\mathbb{D}^n\to\mathbb{D}^n$ with $v(z_1, z_2, \ldots, z_n)=(v_1(z_1), v_2(z_2),\ldots, v_n(z_n))$ be an analytic self-map of $\mathbb{D}^n$ such that  $W_{u, v, k_1, k_2, \ldots, k_n}$ is bounded on $H^2(\mathbb{D}^n)$. If $W_{u, v, k_1, k_2, \ldots, k_n}$ is complex symmetric with standard conjugation $\mathcal{J}$, then $ker(W_{u, v, k_1, k_2, \ldots, k_n})=ker(W^*_{u, v, k_1, k_2, \ldots, k_n})=\mathcal{P}^n_{m-1}(\mathbb{C})$.	
\end{thm}
\begin{proof}[\bf Proof of Theorem \ref{thm-5.1}]
	Suppose that $W_{u, v, k_1, k_2, \ldots, k_n}$ is complex symmetric on $H^2(\mathbb{D}^n)$ with the standard conjugation $\mathcal{J}$. Then
	\begin{align}\label{Eq-5.1}
		W_{u, v, k_1, k_2, \ldots, k_n}\mathcal{J}f(z)=\mathcal{J}W^*_{u, v, k_1, k_2, \ldots, k_n}f(z),\; f\in\ H^2(\mathbb{D}^n).
	\end{align}
 If $f\in\ ker(W_{u, v, k_1, k_2, \ldots, k_n})$, then 
	\begin{align*}
		W_{u, v, k_1, k_2, \ldots, k_n}(f(z))\equiv 0\; \mbox{on}\; \mathbb{D}^n.
	\end{align*}
	That is,
	\begin{align*}
		 u(z)\frac{\partial^{m}f}{\partial z^{k_1}_1\partial z^{k_2}_2 \ldots \partial z^{k_n}_n}(v(z))\equiv 0\; \mbox{on}\; \mathbb{D}^n,
	\end{align*}
	equivalently, we have
	\begin{align*}
	\frac{\partial^{m}f}{\partial z^{k_1}_1\partial z^{k_2}_2 \ldots \partial z^{k_n}_n}(v(z))\equiv 0\; \mbox{on}\; \mathbb{D}^n.
	\end{align*}
	Since $v(\mathbb{D}^n)$ is open, so by Identity Theorem in several variables, we have 
	\begin{align*}
	\frac{\partial^{m}}{\partial z^{k_1}_1\partial z^{k_2}_2 \ldots \partial z^{k_n}_n}(f(z))\equiv 0\; \mbox{on}\; \mathbb{D}^n
	\end{align*}
which implies that
\begin{align*}
	f(z)=\sum_{a_1+a_2+\cdots+a_n\leq (m-1)}A_{a_1, a_2, \ldots a_n}\prod_{i=1}^{n}z^{a_i}_i
\end{align*}
	for some $A_{a_1, a_2, \dots a_n}(a_1+a_2+\cdots+a_n\leq (m-1))\in\mathbb{C} $. Clearly,
	\begin{align}\label{Eq-5.2}
		ker(W_{u, v, k_1, k_2, \ldots, k_n})\subseteq\mathcal{P}^n_{m-1}(\mathbb{C}).
	\end{align}
Let $ f $ be an element of $ \mathcal{P}^n_{m-1}(\mathbb{C}) $. It is clear that $ f $ is a polynomial of degree $\leq(m-1)$ in $n$ variables. Thus, by applying the Identity Theorem in several variables, we have
	\begin{align*}
	W_{u, v, k_1, k_2, \ldots, k_n}(f(z))&=u(z)\frac{\partial^{m}f}{\partial z^{k_1}_1\partial z^{k_2}_2 \ldots \partial z^{k_n}_n}(v(z))\\&\equiv u(z)\frac{\partial^{m}}{\partial z^{k_1}_1\partial z^{k_2}_2 \ldots \partial z^{k_n}_n}(f(z))\; \mbox{on}\; \mathbb{D}^n\\&\equiv 0\; \mbox{on}\; \mathbb{D}^n
	\end{align*}
	which turns out that $f\in\ ker(W_{u, v, k_1, k_2, \ldots, k_n})$. Thus, we have
	\begin{align}\label{Eq-5.3}
		\mathcal{P}^n_{m-1}(\mathbb{C})\subseteq\ker(W_{u, v, k_1, k_2, \ldots, k_n}).
	\end{align}
	Hence, based on \eqref{Eq-5.2} and \eqref{Eq-5.3}, we conclude that $ ker(W_{u, v, k_1, k_2, \ldots, k_n})=\mathcal{P}^n_{m-1}(\mathbb{C})$. Furthermore, for $f\in\ ker(W^*_{u, v, k_1, k_2, \ldots, k_n})$, we have $	W^*_{u, v, k_1, k_2, \ldots, k_n}(f)=0 $. Since $ D_{n,\psi, \phi} $ is complex symmetric on $H^2(\mathbb{D}^n)$ with conjugation $\mathcal{J}$, we have from \eqref{Eq-5.1} that
	\begin{align*}
		&W_{u, v, k_1, k_2, \ldots, k_n}\mathcal{J}f=\mathcal{J}W^*_{u, v, k_1, k_2, \ldots, k_n}f=0.
	\end{align*}
	That is,
	\begin{align*}
		\mathcal{J}f\in\ker(W_{u, v, k_1, k_2, \ldots, k_n})=\mathcal{P}^n_{m-1}(\mathbb{C}),
	\end{align*}
	equivalently, we have
	\begin{align*}
    f\in\mathcal{P}^n_{m-1}(\mathbb{C}).	
	\end{align*}
	Hence,
	\begin{align}\label{Eq-5.4}
		ker(W^*_{u, v, k_1, k_2, \ldots, k_n})\subseteq\mathcal{P}^n_{m-1}(\mathbb{C}). 	
	\end{align} 
	Let $g\in\mathcal{P}^n_{m-1}(\mathbb{C})$. Then, we have $W^*_{u, v, k_1, k_2, \ldots, k_n}g=\mathcal{J}W_{u, v, k_1, k_2, \ldots, k_n}\mathcal{J}g=0 $ which implies that $g\in\ker(W^*_{u, v, k_1, k_2, \ldots, k_n})$. Hence,
	\begin{align}\label{Eq-5.5}
		\mathcal{P}^n_{m-1}(\mathbb{C})\subseteq ker(W^*_{u, v, k_1, k_2, \ldots, k_n}). 	
	\end{align}
	Consequently, from \eqref{Eq-5.4} and \eqref{Eq-5.5}, we obtain $ ker(W^*_{u, v, k_1, k_2, \ldots, k_n})=\mathcal{P}^n_{m-1}(\mathbb{C}) $. Therefore,
	\begin{align*}
		ker(W_{u, v, k_1, k_2, \ldots, k_n})=ker(W^*_{u, v, k_1, k_2, \ldots, k_n})=\mathcal{P}^n_{m-1}(\mathbb{C}).	
	\end{align*}
	This completes the proof.
\end{proof}
\section{Concluding remarks}
In the article, we have discussed the $\mathcal{J}$-symmetric exhibited by the partial composition-differentiation operator $W_{u, v, k_1, k_2, \ldots, k_n}$. Our research has involved investigating normality and Hermitian characteristics of these operators. Additionally, we have determined the kernels of the operators $W_{u, v, k_1, k_2, \ldots, k_n}$and their adjoints. furthermore, we have observed that the discussion on normality can be extended to encompass scenarios where $v_j(0)\neq0(j=1, 2, \ldots, n)$. Moreover, the normality and symmetry of these operators can be discussed independently using different conjugation methods. 

\section{Declaration}
\noindent\textbf{Compliance of Ethical Standards}\\

\noindent\textbf{Conflict of interest} The authors declare that there is no conflict  of interest regarding the publication of this paper.\vspace{1.5mm}

\noindent\textbf{Data availability statement}  Data sharing not applicable to this article as no datasets were generated or analysed during the current study.


\begin{thebibliography}{200}
	




\bibitem{Allu-Hal-Pal-2023} {\sc V. Allu, H. Haldar}, and {\sc S. Pal}, Composition-differentiation operator on weighted Bergman spaces,  (2023) arXiv:2301.08575v1.







\bibitem{Bourdon-2002} {\sc P. S. Bourdon}, Some adjoint formulas for composition operators on $ H^2 $, unpublished manuscript, 2002.

\bibitem{Bourdon-cima-matheson-TAMS-1999} {\sc P. S. Bourdon}, {\sc J. A. Cima}, and {\sc A. L. Matheson}, Compact composition operators on BMOA, {\it Trans. Amer. Math. Soc.} {\bf 351}(1999), 2183–2196.

\bibitem{Bourdon-Fry-Hammond-Spofford-TAMS-2004} {\sc P. S. Bourdon}, {\sc E. E. Fry}, {\sc C. Hammond}, and {\sc C. H. Spofford}, Norms of linear-fractional composition operators, {\it Trans. Amer. Math. Soc.} {\bf 356}(2004), 2459–2480.

\bibitem{Bourdon-Levi-Narayan-Shapiro-JMAA-2003} {\sc P. S. Bourdon}, {\sc D. Levi}, {\sc S. K. Narayan}, and {\sc J. H. Shapiro}, Which linear-fractional composition operators are essentially normal? {\it J. Math. Anal. Appl.} {\bf 280}(2003), 30-53.

\bibitem{Bourdon-MacCluer-CVEE-2007} {\sc P. S. Bourdon} and {\sc  B. D. MacCluer}, Selfcommutators of automorphic composition operators, {\it  Complex Var. Elliptic Equ.} {\bf 52}(2007), 85-104.

\bibitem{Bourdon-Narayan-JMAA-2010} {\sc P. S. Bourdon} and {\sc S. K. Narayan}, Normal weighted composition operators on the Hardy space $H^2(\mathbb{U})$, {\it J. Math. Anal. Appl.} {\bf 367} (2010), 278–286.

\bibitem{Bourdon-Noor-JMAA-2015} {\sc P. Bourdon} and {\sc S. Noor}, Complex symmetry of invertible composition operators, \textit{J. Math. Anal. Appl.} \textbf{429}(2015), 105–110.






\bibitem{Clifford-Zheng-IUMJ-1999} {\sc J. Clifford} and {\sc D. Zheng}, Composition operators on the Hardy space, {\it Indiana Univ. Math. J.} {\bf 48}(1999), 1585-1616.











\bibitem{Cowen-Gallardo-JFA-2006} {\sc C. C. Cowen} and {\sc E. A. Gallardo-Guti\`errez}, A new class of operators and a description of adjoints of composition operators, {\it J. Funct. Anal.} {\bf 238}(2006), 447-462.

\bibitem{Cowen-MacCluer-CRCP-1995} {\sc C. C. Cowen} and {\sc B. D. MacCluer}, Composition Operators on Spaces of Analytic Functions, {\it CRC Press, Boca Raton}, (1995).

\bibitem{Cowen-MacCluer} {\sc C. C. Cowen} and {\sc B. D. MacCluer}, Some problems on composition operators, in: F. Jafari, B.D. MacCluer, C. Cowen, D. Porter (Eds.), Studies of Composition Operators, Laramie, WY, 1996, in: Contemp. Math., vol. 213, Amer. Math. Soc., Providence, RI, 1998, pp. 17–25.





\bibitem{Duren-APNY-1970} {\sc P. L. Duren}, Theory of $ H^p $ Spaces, {\it Academic Press, New York}, 1970.






\bibitem{Fatehi-Hammond-CAOT-2021} {\sc M. Fatehi} and {\sc C. Hammond}, Normality and self-adjointness of weighted composition-differentiation operators, \textit{Complex Anal. Oper. Theory} \textbf{15}(1)(2021), 13.

\bibitem{Fatehi-Moradi-2021} {\sc M. Fatehi} and {\sc M. Moradi}, Complex symmetric Weighted Composition-Differentiation Operators of order $n$ on the Weighted Bergman Spaces, 
https://doi.org/10.48550/arXiv.2101.04911.




\bibitem{Garcia-Hammond-OPAP-2013} {\sc S. Garcia} and {\sc C. Hammond}, Which weighted composition operators are complex symmetric? \textit{Oper. Theory Adv. Appl.} \textbf{236}(2013), 171–179.

\bibitem{Garcia-Putinar-TAMS-2006} {\sc S. Garcia} and {\sc M. Putinar}, Complex symmetric operators and applications, {\it Trans. Amer. Math. Soc.} {\bf 358}, (2006), 1285–1315.

\bibitem{Garcia-Putinar-TAMS-2007} {\sc S. Garcia} and {\sc M. Putinar}, Complex symmetric operators and applications II, {\it Trans. Amer. Math. Soc.} {\bf 359}(2007), 3913–3931.

\bibitem{Garcia-Putinar-TAMS-2008} {\sc S. Garcia} and {\sc M. Putinar}, Interpolation and complex symmetry, {\it Tohoku Math. J.} {\bf 60}(2008), 423–440.

\bibitem{Garcia-Wogen-JFA-2009} {\sc S. Garcia} and {\sc W. Wogen}, Complex symmetric partial isometries, {\it J. Funct. Anal.} {\bf 257}(2009), 1251–1260.

\bibitem{Garcia-Wogen-TAMS-2010} {\sc S. Garcia} and {\sc W. Wogen}, Some new classes of complex symmetric operators, {\it Trans. Amer. Math. Soc.} {\bf 362}(2010), 6065–6077. 

\bibitem{Garnett-APNY-1981} {\sc J. B. Garnett}, Bounded Analytic Functions, {\it Academic Press, New York}, (1981).









\bibitem{Halmos-DNPNJ-1967} {\sc P. R. Halmos}, A Hilbert Space Problem Book, {\it D. van Nostrand, Princeton, NJ}, (1967).

\bibitem{Halmos-SVNY-1982}{\sc P. R. Halmos}, A Hilbert Space Problem Book, {\it 2nd ed., Springer-Verlag, New York}, (1982).



\bibitem{Han-Wang-JMAA-2019} {\sc K. K. Han} and {\sc M. Wang}, Complex symmetric weighted composition operators in several variables, {\it J. Math. Anal. Appl} {\bf 474} (2019), 961-987.

\bibitem{Han-Wang-BJMA-2021} {\sc K. K. Han} and {\sc M. Wang}, Weighted composition–differentiation operators on the Hardy space, \textit{Banach J. Math. Anal.} (2021) 15:44.





\bibitem{Han-Wang-Wu-AMS-2023} {\sc K. K. Han}, {\sc M. Wang}, and {\sc Q. Wu}, Complex Symmetry of Toeplitz Operators over the Bidisk, {\it Acta Mathematica Scientia} {\bf 43B}(4)(2023), 1537-1546.








\bibitem{Hu-Yang-Zhou-IJM-2020} {\sc X. H. Hu}, {\sc Z. C. Yang}, and {\sc Z. H. Zhou}, Complex symmetric weighted composition operators on Dirichlet spaces and Hardy spaces in the unit ball, {\it Inter. J. Math.} {\bf 31} (2020) 2050006.










\bibitem{Jun-Kim-Ko-Lee-JFA-2014} {\sc S. Jung, Y. Kim, E. Ko}, and {\sc J. E. Lee}, Complex symmetric weighted composition operators on ${H}^2(\mathbb{D})$, \textit{J. Func. Anal.} \textbf{267}(2014) 323–351.







\bibitem{Le-JMAA-2012} {\sc T. Le}, Self-adjoint, unitary, and normal weighted composition operators in several variables, {\it J. Math. Anal. Appl} {\bf 395} (2012), 596-607.








\bibitem{Lim-Khoi-JMAA-2018} {\sc R. Lim} and {\sc L. Khoi}, Complex symmetric weighted composition operators on $ \mathcal{H}_{\gamma}(\mathbb{D}) $, \textit{J. Math. Anal. Appl.} \textbf{464}(2018), 101-118.




\bibitem{Liu-Ponnusamy-Xie-LMA-2023} {\sc J. Liu}, {\sc S. Ponnusamy}, and {\sc H. Xie}, Complex symmetric weighted composition-differentiation operators, {\it Linear and Multilinear Algebra}, {\bf 71}(2023), 737-755.



\bibitem{Lo-Loh-JMAA-2023} {\sc C. Lo} and {\sc A. W.-k Loh}, Complex symmetric generalized weighted composition operators on Hilbert spaces of analytic functions, \textit{J. Math. Anal. Appl.} \textbf{523}(2023) 127141.






























\bibitem{Noor-Severiano-PAMS-2020} {\sc S. Noor} and {\sc O. Severiano}, Complex symmetry and cyclicity of composition operators on $ H^2(\mathbb{C}_{+}) $, \textit{Proc. Amer. Math. Soc.} 148(2020), 2469-2476.



\bibitem{Ohno-BAMS-2006} {\sc S. Ohno}, Products of composition and differentiation between Hardy spaces, \textit{Bull. Austral. Math. Soc.} \textbf{73}(2006), 235-243.












\bibitem{Shapiro-SVNY-1993} {\sc J. H. Shapiro}, Composition Operators and Classical Function Theory, {\it Springer-Verlag, New York}, (1993).













\bibitem{Wang-Yao-IJM-2016} {\sc M. F. Wang} and {\sc X. X. Yao}, Complex symmetry of weighted composition operators in several variables, {\it Inter. J. Math.} {\bf 27} (2016) 1650017.







\bibitem{Yao-JMAA-2017} {\sc X. Yao}, Complex symmetric composition operators on a Hilbert space of Dirichlet series, \textit{J. Math.
	Anal. Appl.} \textbf{452}(2017), 1413-1419.






\end{thebibliography}
\end{document}